\documentclass[12pt, reqno]{amsart}
\usepackage{mathtools, amscd, amsfonts, amssymb}
\usepackage{bm}
\usepackage{graphicx}
\usepackage[shortlabels]{enumitem}
\usepackage{mathrsfs}
\usepackage{color}
\usepackage{cite}
\usepackage{lineno}
\usepackage[
colorlinks=true,
linkcolor=blue,
citecolor=blue,
urlcolor=blue,
pagebackref=true
]{hyperref}
\newtheorem{theorem}{Theorem}[section]
\newtheorem{lemma}[theorem]{Lemma}
\newtheorem{proposition}[theorem]{Proposition}
\newtheorem{corollary}[theorem]{Corollary}
\theoremstyle{definition}
\newtheorem{definition}[theorem]{Definition}
\newtheorem{example}[theorem]{Example}

\theoremstyle{remark}
\newtheorem{remark}[theorem]{Remark}
\numberwithin{equation}{section}

\numberwithin{equation}{section}

\DeclareMathOperator{\im}{range}

\usepackage[
colorlinks=true,
linkcolor=blue,
citecolor=blue,
urlcolor=blue,
pagebackref=true
]{hyperref}

\usepackage{etoolbox}
\makeatletter
\patchcmd{\@settitle}{\uppercasenonmath\@title}{}{}{}
\patchcmd{\@settitle}{\MakeUppercase}{\@firstofone}{}{}

\patchcmd{\@setauthors}{\MakeUppercase}{\@firstofone}{}{}
\patchcmd{\@setauthors}{\uppercasenonmath}{\@firstofone}{}{}

\@ifpackageloaded{hyperref}{%
	\patchcmd{\HyOrg@maketitle}{\uppercasenonmath\shorttitle}{}{}{}%
	\patchcmd{\HyOrg@maketitle}{\uppercasenonmath\shortauthors}{}{}{}%
	\patchcmd{\HyOrg@maketitle}{\@nx\MakeUppercase{\the\toks@}}{\the\toks@}{}{}%
	\patchcmd{\HyOrg@maketitle}{\@nx\MakeUppercase{\the\@temptokena}}{\the\@temptokena}{}{}%
}{%
	\patchcmd{\maketitle}{\uppercasenonmath\shorttitle}{}{}{}%
	\patchcmd{\maketitle}{\uppercasenonmath\shortauthors}{}{}{}%
	\patchcmd{\maketitle}{\@nx\MakeUppercase{\the\toks@}}{\the\toks@}{}{}%
	\patchcmd{\maketitle}{\@nx\MakeUppercase{\the\@temptokena}}{\the\@temptokena}{}{}%
}

\patchcmd{\section}{\scshape}{\bfseries}{}{}

\def\@secnumfont{\bfseries}

\patchcmd{\@setaddresses}{\scshape}{\upshape}{}{}
\makeatother
\begin{document}
	
	\title[Regularity of the Factorization of the Characteristic Function]{Regularity of the Sz.-Nagy and Foia\c{s} Factorization of Characteristic Functions and Its Multivariable Analogue}

	\author[Kalpesh J. Haria]{Kalpesh J. Haria}
	
	\address{School of Mathematics and Computer Science,
		Indian Institute of Technology Goa,
		Goa 403401, India}
	\email{kalpesh@iitgoa.ac.in, hikalpesh.haria@gmail.com}
	
	\author[Aashish Kumar Maurya]{Aashish Kumar Maurya}

	\address{School of Mathematics and Computer Science,
		Indian Institute of Technology Goa,
		Goa 403401, India}
	\email{aashish21232101@iitgoa.ac.in, a.k.maurya.math@gmail.com}

	% Subject classification and keywords - MODIFY THIS SECTION
	
	% Subject classification and keywords

\begin{abstract}
	Given a contraction with an invariant subspace and a row contraction with a joint invariant subspace, a factorization of their characteristic functions was obtained by Sz.-Nagy and Foia\c{s}, and by Haria, Maji, and Sarkar, respectively. In this article, we investigate the regularity of this factorization in both the single-variable and multivariable cases. We construct examples of a contraction $T$ with an invariant subspace such that this factorization is not regular in general. If $T$ is completely non-unitary, we prove that this factorization is either regular or strange. Furthermore, we obtain a characterization of the regularity of this factorization. Using this characterization, we identify several classes of contractions and row contractions for which this factorization is regular. Among these classes, the most notable classes are pure contractions and pure row contractions. Additionally, for any integer $k>2$, we introduce the concept of $k$-strange factorizations for contractive analytic functions, which extend the concept of strange ($2$-strange) factorizations introduced by Sz.-Nagy and Foia\c{s}. Finally, we prove that if a contraction is pure or is a completely non-unitary contraction for which $\Delta_{\Theta_T}(t)$ has finite rank almost everywhere, then its characteristic function does not admit any $k$-strange factorization.
\end{abstract}

	\maketitle
	
	\tableofcontents

	% Ensure \usepackage{xcolor} is included in your document preamble to compile \textcolor.
	
	\section{Introduction}
	
	Let $\mathcal{H}$ be a separable Hilbert space over the complex field, and let $T\in B(\mathcal{H})$ be a contraction. Associated with every contraction $T$, Sz.-Nagy and Foia\c{s} introduced a purely contractive operator-valued analytic function, called the \emph{characteristic function} of $T$, and established that it serves as a complete unitary invariant for completely non-unitary contractions. A comprehensive exposition of this theory can be found in the monograph \cite{NFBK10}.
	
	Let $T\in B(\mathcal{H})$ be a bounded linear operator on a separable complex Hilbert space $\mathcal{H}$, and suppose that $\mathcal{H}_1\subseteq\mathcal{H}$ is a closed invariant subspace for $T$. Setting $\mathcal{H}_2 = \mathcal{H} \ominus \mathcal{H}_1$, the operator $T$ can be decomposed with respect to the orthogonal decomposition $\mathcal{H}_1\oplus\mathcal{H}_2$ into the block upper-triangular form
	\begin{align}\label{triangular_form}
		T=\begin{bmatrix}
			A & X\\
			0 & B
		\end{bmatrix}:\mathcal{H}_1\oplus\mathcal{H}_2\to \mathcal{H}_1\oplus\mathcal{H}_2.
	\end{align}
	In Theorem~1 of \cite{NF67}, Sz.-Nagy and Foia\c{s} established necessary and sufficient conditions for an operator $T$ of this form to be a contraction.
	
	For a contraction $T\in B(\mathcal{H})$ of the form \eqref{triangular_form}, Sz.-Nagy and Foia\c{s} proved in \cite{NF67} that the characteristic function $\Theta_T$ admits a factorization into a product of contractive analytic functions:
	\begin{align}\label{main_fact_1}
		\Theta_T(z) =
		\underbrace{
			\tau_*^{-1}
			\begin{bmatrix}
				\Theta_{B}(z) & 0 \\
				0 & I_{\mathcal{D}_{L^*}}
			\end{bmatrix}
		}_{\text{third factor}}
		\underbrace{
			\begin{bmatrix}
				L^* & D_L \\
				D_{L^*} & -L
			\end{bmatrix}
		}_{\text{second factor}}
		\underbrace{
			\begin{bmatrix}
				\Theta_{A}(z) & 0 \\
				0 & I_{\mathcal{D}_L}
			\end{bmatrix} \tau
		}_{\text{first factor}}
		\quad \text{for all } z \in \mathbb{D}.
	\end{align}
	Furthermore, after decomposing each factor into its constant unitary and purely contractive parts, the purely contractive parts of the first and third factors are the characteristic functions of the contractions $A$ and $B$, respectively, while the middle factor is the Julia--Halmos matrix corresponding to the coupling operator $L$.
	
	In the multivariable setting, Haria, Maji, and Sarkar \cite{HMS17} generalized this factorization theorem to row contractions possessing a joint invariant subspace, showing that the characteristic function $\Theta_T$ admits the decomposition
	\begin{align}\label{main_fact_2}
		\Theta_T = 
		\underbrace{
			(I_\Gamma \otimes \sigma_*^{-1})
			\begin{bmatrix}
				\Theta_B & 0 \\
				0 & I_{\Gamma \otimes \mathcal{D}_{L^*}}
			\end{bmatrix}
		}_{\text{third factor}}
		\underbrace{
			(I_\Gamma \otimes J_L)
		}_{\text{second factor}} 
		\underbrace{
			\begin{bmatrix}
				\Theta_A & 0 \\ 
				0 & I_{\Gamma \otimes \mathcal{D}_{L}}
			\end{bmatrix}
			(I_\Gamma \otimes \sigma)
		}_{\text{first factor}},
	\end{align}
	where the purely contractive parts of the first and third factors coincide with the characteristic functions of the row contractions $A$ and $B$, respectively.
	
	As established by Sz.-Nagy and Foia\c{s} \cite{NF64,N70} in the single-operator setting and by Popescu \cite{Po06} in the multivariable setting, whenever a completely non-unitary contraction admits a closed invariant subspace, or a completely non-coisometric row contraction admits a joint invariant subspace, its characteristic function admits a regular factorization. In this regular factorization, the purely contractive parts of the first and second factors are coincide with the characteristic functions of the diagonal blocks $A$ and $B$, respectively.
	
	Motivated by these observations, in \cite{HMS17}, Haria, Maji, and Sarkar raised the question  of how the factorizations \eqref{main_fact_1} and \eqref{main_fact_2} are related to the regular factorizations of the characteristic functions of a contraction and a row contraction, respectively.

	The main goal of the present paper is to determine whether the factorizations
	\eqref{main_fact_1} and \eqref{main_fact_2} of the characteristic function are
	$3$-regular. This problem is motivated by the notion of $k$-regular
	factorization, introduced in \cite{HM26a} for products of contractions with
	$k>2$, as a natural extension of the regular factorization of Sz.-Nagy and
	Foia\c{s}.

	The theory of regular factorizations has been extensively employed in the study of invariant and hyperinvariant subspaces. Notable contributions in this direction include the works of P.~Y.~Wu \cite{Wu78a,Wu78b,Wu79a,Wu79b} and the series of papers by R.~I.~Teodorescu \cite{Te75,Te76,Te77,Te78a,Te78b,Te79,Te80}. More recently, D.~Timotin \cite{T20} applied the notion of regular factorization to characterize all invariant subspaces of the operator $S \oplus S^*$, where $S$ denotes the unilateral shift on the Hardy space $H^2(\mathbb{D})$.
	
	Section~2 is devoted to preliminary results and definitions required in subsequent sections. In Section~3, we introduce the concept of $k$-strange factorizations for contractive analytic functions, which generalizes the classical notion of strange factorizations given in Chapter~VII of \cite{NFBK10}. We first prove that if a contraction $T$ is pure, or if $T$ is completely non-unitary (c.n.u.) and the defect operator $\Delta_{\Theta_{T}}(t)$ has finite rank for almost every $t\in [0,2\pi]$, then its characteristic function $\Theta_T$ admits no $k$-strange factorization. Next, for a c.n.u.\ contraction $T\in B(\mathcal{H})$ in the block upper-triangular form \eqref{triangular_form}, we show that the factorization \eqref{main_fact_1} is either $3$-regular or $3$-strange; consequently, if the rank of $\Delta_{\Theta_{T}}(t)$ is finite almost everywhere, then \eqref{main_fact_1} is necessarily $3$-regular.
	
	Furthermore, we obtain a characterization of the $3$-regularity of \eqref{main_fact_1} and apply it to identify several classes of contractions for which this factorization is $3$-regular, including a distinguished class of pure contractions. Whether $3$-regularity holds for arbitrary c.n.u.\ contractions currently remains an open question.
	
	In Section~4, we present counterexamples of contractions in block upper-triangular form for which the factorization \eqref{main_fact_1} fails to be $3$-regular. Notably, none of the constructed counterexamples is a c.n.u.\ contraction; consequently, the problem remains open for arbitrary c.n.u.\ contractions.
	
	Section~5 is devoted to the study of factorizations of characteristic functions associated with row contractions. The principal aim of this section is to examine the regularity of the factorization \eqref{main_fact_2} introduced in \cite{HMS17}. To this end, we first recall the necessary definitions regarding characteristic functions and regular factorizations in the multivariable setting. As demonstrated in Section~4, the factorization \eqref{main_fact_2} is not $3$-regular in general. Nevertheless, we identify several classes of row contractions for which this factorization is $3$-regular. In particular, regularity is established for a distinguished class of pure row contractions; whether $3$-regularity holds for arbitrary completely non-coisometric row contractions remains an open question.

	\section{Preliminaries}\label{sec prelim}

	In this section, we recall the fundamental definitions and preliminary results that will be utilized throughout this paper. All Hilbert spaces considered herein are assumed to be separable and defined over the complex field. For Hilbert spaces \(\mathcal{H}\) and \(\mathcal{K}\), \(B(\mathcal{H}, \mathcal{K})\) denotes the space of all bounded linear operators from \(\mathcal{H}\) to \(\mathcal{K}\).

	\begin{definition}
		Let \(T\in B(\mathcal H,\mathcal K)\). The \emph{operator norm} of \(T\) is defined by
		\[
		\|T\|\coloneq\sup\{\|Tx\|:\,x\in\mathcal H,\ \|x\|\le1\}.
		\]
		The operator \(T\) is said to be:
		\begin{enumerate}[\rm(i)]
			\item a \emph{contraction} if \(\|Tx\|\le\|x\|\) for all \(x\in\mathcal H\);
			
			\item a \emph{proper contraction} if \(\|Tx\|<\|x\|\) for all nonzero \(x\in\mathcal H\);
			
			\item a \emph{strict contraction} if \(\|T\|<1\);
			
			\item an \emph{isometry} if \(\|Tx\|=\|x\|\) for all \(x\in\mathcal H\); 
			
			\item \emph{unitary} if it is a surjective isometry;
			
			\item \emph{pure} if \(\mathcal{H}=\mathcal{K}\) and \(\lim\limits_{n\to\infty}\|T^{*n}x\|=0\) for all \(x\in\mathcal H\).
		\end{enumerate}
	\end{definition}

	\noindent Given a contraction \(T \in B(\mathcal{H}, \mathcal{K})\), its associated \emph{defect operators} are defined by  
	\[
	D_T := (I - T^*T)^{1/2}, \quad D_{T^*} := (I - TT^*)^{1/2}.
	\]  
	The corresponding \emph{defect spaces} are defined as the closures of the ranges of these operators, namely  
	\[
	\mathcal{D}_T := \overline{range(D_T)}, \quad \mathcal{D}_{T^*} := \overline{range(D_{T^*})},
	\]  
	and their respective dimensions, denoted by \(\partial_T := \dim \mathcal{D}_T\) and \(\partial_{T^*} := \dim \mathcal{D}_{T^*}\), are referred to as the \emph{defect indices} of \(T\).
	
	Suppose $\mathcal{H}_1$ and $\mathcal{H}_2$ are Hilbert spaces, and let $T \in B(\mathcal{H}_1 \oplus \mathcal{H}_2)$ be a bounded linear operator for which $\mathcal{H}_1$ is an invariant subspace. The following theorem establishes a necessary and sufficient condition for the corresponding upper triangular block matrix representation of \(T\) to be a contraction.
	\begin{theorem}[\rm Theorem 1, \cite{NF67}]\label{lemma:1}
		Let \( \mathcal{H}_1 \) and \( \mathcal{H}_2 \) be Hilbert spaces.  
		Let \( T = \begin{bmatrix} A & X \\ 0 & B \end{bmatrix} \) be a contraction on the Hilbert space \( \mathcal{H}_1 \oplus \mathcal{H}_2 \). Then there exists a contraction \( L\in B(\mathcal{D}_B ,\mathcal{D}_{A^*}) \), called the \emph{coupling operator}, such that  
		\[
		X = D_{A^*} L D_B.
		\]  
		Conversely, if \( A \in B(\mathcal{H}_1) \), \( B \in B(\mathcal{H}_2) \), and \( L\in B(\mathcal{D}_B ,\mathcal{D}_{A^*}) \) are contractions, then  
		\[
		T = \begin{bmatrix} A & D_{A^*} L D_B \\ 0 & B \end{bmatrix}
		\]  
		is a contraction on the Hilbert space \( \mathcal{H}_1 \oplus \mathcal{H}_2 \).
	\end{theorem}
	A contraction \(T\) on \(\mathcal{H}\) is called \emph{completely nonunitary }(c.n.u.)\ if there is no reducing subspace on which \(T\) acts unitarily. If $T = \begin{bmatrix} A & D_{A^*} L D_B \\ 0 & B \end{bmatrix}\in B(\mathcal{H}_1\oplus\mathcal{H}_2)$ is a c.n.u.\ contraction, then the operators $A\in B(\mathcal{H}_1)$ and $B\in B(\mathcal{H}_2)$ are necessarily c.n.u.\ contractions. However, the converse does not hold in general; that is, there exist c.n.u.\ contractions $A\in B(\mathcal{H}_1)$ and $B\in B(\mathcal{H}_2)$ such that $T = \begin{bmatrix} A & D_{A^*} L D_B \\ 0 & B \end{bmatrix}\in B(\mathcal{H}_1\oplus\mathcal{H}_2)$ is not a c.n.u.\ contraction, as demonstrated in \cite{NF67}. Nevertheless, a simple condition on the coupling operator $L$ guarantees the converse, as stated in the following lemma.
	
	\begin{lemma}[\rm p.~204, \cite{NF67}]\label{lemma:2}
		Let \( \mathcal{H}_1 \) and \( \mathcal{H}_2 \) be Hilbert spaces, and suppose that \( A\in B(\mathcal{H}_1) \) and \( B\in B(\mathcal{H}_2) \) are c.n.u.\ contractions. If there exists a contraction \( L\in B(\mathcal{D}_B ,\mathcal{D}_{A^*}) \) such that 
		\begin{align}\label{condition_L}
			\|Lg\| < \|g\| ~ \text{for all } g\in {D}_B \mathcal{H}_2\setminus\{0\}, 
			~\text{and} ~
			\|L^*f\| < \|f\| ~ \text{for all } f\in {D}_{A^*} \mathcal{H}_1\setminus\{0\},
		\end{align}
		then the operator
		\[
		T= \begin{bmatrix} A & D_{A^*} L D_B \\[4pt] 0 & B \end{bmatrix}
		\]
		is a c.n.u.\ contraction on \( \mathcal{H}_1 \oplus \mathcal{H}_2 \).
	\end{lemma}
	
	\begin{remark}
		The condition \eqref{condition_L} imposed on the coupling operator \(L\) in Lemma~\ref{lemma:2} is sufficient, but not necessary, for \(T\) to be a c.n.u.\ contraction. For instance, if \( L \) is an isometry, then the operator
		\[
		T=\begin{bmatrix} 0 & L \\[4pt] 0 & 0 \end{bmatrix}
		\]
		is nilpotent and therefore pure, which implies that \(T\) is a c.n.u.\ contraction. However, because $L$ is isometric, it fails to satisfy condition \eqref{condition_L}.
	\end{remark}

	We recall the following basic definitions and notations from Chapters V and VI of \cite{NFBK10}. Let \(\mathcal{E}\) and \(\mathcal{E}_*\) be Hilbert spaces, and let \(\Theta:\mathbb{D}\to B(\mathcal{E},\mathcal{E}_*)\) be an operator-valued function. The triple \(\{\mathcal{E},\mathcal{E}_*,\Theta(z)\}\) is said to be a \emph{bounded analytic function} if the following conditions are satisfied:
	
   \noindent (i) The power series expansion
		\[
		\Theta(z) = \sum_{k=0}^{\infty} z^k \Theta_k, \quad z \in \mathbb{D}, \ \Theta_k \in B(\mathcal{E}, \mathcal{E}_*)
		\]
		converges on the open unit disc \(\mathbb{D}\) (note that in this context, convergence in the weak, strong, and norm topologies is equivalent).
		
		\noindent(ii) There exists a constant \(M\), independent of \(z\), such that \(\|\Theta(z)\|\leq M\) for all \(z \in \mathbb{D}\). 
		
		A bounded analytic function \(\{\mathcal{E}, \mathcal{E}_*, \Theta(z)\}\) is termed a \emph{contractive analytic function} if  
	\(\|\Theta(z)\|\leq 1 ~ \text{for all } z \in \mathbb{D}.\)
	Furthermore, the function \(\{\mathcal{E}, \mathcal{E}_*, \Theta(z)\}\) is said to be \emph{purely contractive} if \(\|\Theta(0)e\| < \|e\| \quad \text{for all nonzero } e \in \mathcal{E}.\)
	
	\begin{definition}[Chapter~VI, \cite{NFBK10}]
		Let $T$ be a contraction on a Hilbert space $\mathcal{H}$. The \emph{characteristic function} of $T$ is the operator-valued analytic function $\Theta_T:\mathbb{D}\to B(\mathcal{D}_T, \mathcal{D}_{T^*})$ defined by
		\[
		\Theta_T(z) := (-T+zD_{T^*}(I-zT^*)^{-1}D_T)|_{\mathcal{D}_T}, \quad \text{for all } z\in\mathbb{D}.
		\]
	\end{definition}
	For any contraction $ T\in B(\mathcal H),$ the restriction operator $ T|_{\mathcal{D}_T}: \mathcal{D}_T\to\mathcal{D}_{T^*} $ is a proper contraction. Which implies that for an arbitrary contraction $ T $, the characteristic function $\{\mathcal{D}_T,\mathcal{D}_{T^*},\Theta_{T}(z)\}$ is a purely contractive analytic function. 
	
	Let \(\{\mathcal{E}, \mathcal{E_*}, \Theta(z)\}\) be a bounded analytic function. We can associate the operator \(\Theta: L^2(\mathcal{E}) \to L^2(\mathcal{E_*})\) defined by 
	\begin{align}\label{induce_contraction}
		(\Theta v)(t) := \Theta(e^{it}) v(t) \quad \text{for all } v \in L^2(\mathcal{E}),
	\end{align}
	where \(t \in [0, 2\pi]\) and \(\Theta(e^{it})=\displaystyle\lim_{r\to 0^-}\Theta(re^{it})\) exists a.e.\ as a strong operator limit. We also define the operators 
	\[
	\Delta_\Theta(t) := (I_{\mathcal{E}} - \Theta^*(e^{it})\Theta(e^{it}))^{\frac{1}{2}} \text{   and } 	\Delta_{*\Theta}(t) := (I_{\mathcal{E}} - \Theta(e^{it})\Theta(e^{it})^*)^{\frac{1}{2}} .
	\]
	Additionally, we associate another operator  $\Theta_+:H^2(\mathcal{E})\to H^2(\mathcal{E_*})$ defined by 
	\[(\Theta_+u)(z) : = \Theta(z)u(z) ~\text{ for all }~ u\in H^2(\mathcal{E}), ~\text{for all } z \in \mathbb{D}.\]
	A contractive analytic function \(\{\mathcal{E}, \mathcal{E}_*, \Theta(z)\}\) is said to be \emph{inner} if the associated operator \(\Theta_+\) is an isometry. Moreover, a c.n.u.\ contraction \(T\) is pure if and only if its characteristic function is inner.
	\begin{proposition} [Proposition 2.1, Chapter V, \cite{NFBK10}]
		Let \(\{\mathcal{E}, \mathcal{E}_*, \Theta(z)\}\) be a contractive analytic function.  Then there exist uniquely determined orthogonal decompositions $\mathcal{E}=\mathcal{E}^p\oplus\mathcal{E}^u$ and $\mathcal{E}_*=\mathcal{E}_*^p\oplus\mathcal{E}_*^u$ such that
		\begin{align*}
			\Theta(z)=\begin{bmatrix}
				\Theta^p(z) & 0\\
				0 & \Theta^u(z)
			\end{bmatrix} :\mathcal{E}^p\oplus\mathcal{E}^u \to  \mathcal{E}_*^p\oplus\mathcal{E}_*^u,\quad \text{for all } z \in \mathbb{D},
		\end{align*}
		where $\Theta^p(z):=\Theta(z)|_{\mathcal{E}^p}$ is a purely contractive analytic function and $\Theta^u(z):=\Theta(z)|_{\mathcal{E}^u}$ is a unitary constant. The operator $ \{\mathcal{E}^p,\mathcal{E}_*^p,\Theta^p(z)\}$ is called the purely contractive part of  \(\{\mathcal{E}, \mathcal{E}_*, \Theta(z)\}\).
	\end{proposition} 
	
	For a given positive integer \(k \geq 2\), assume that \(\{\mathcal{E}, \mathcal{E}_*, \Theta(z)\}\) and \(\{\mathcal{E}_i, \mathcal{E}_{i+1}, \Theta_i(z)\}\) are contractive analytic functions for $i=1,\dots,k$, where \(\mathcal{E} = \mathcal{E}_1\) and \(\mathcal{E}_* = \mathcal{E}_{k+1}\), such that 
	\begin{align}\label{factorization_theta}
		\Theta(z)=\Theta_k(z)\cdots\Theta_1(z), ~~\text{for all } z \in \mathbb{D}.
	\end{align}
	Let \(\Theta \in B\left(L^2(\mathcal{E}), L^2(\mathcal{E}_*)\right)\) and \(\Theta_i \in B\left(L^2(\mathcal{E}_i), L^2(\mathcal{E}_{i+1})\right)\), for $i=1,\dots,k$, denote the contractions induced by the contractive analytic functions \(\{\mathcal{E}, \mathcal{E}_*, \Theta(z)\}\) and \(\{\mathcal{E}_i, \mathcal{E}_{i+1}, \Theta_i(z)\}\), respectively, as defined in \eqref{induce_contraction}. Following the concept of a \(k\)-regular factorization introduced in \cite{HM26a}, the factorization \eqref{factorization_theta} is called a \emph{\(k\)-regular} factorization if the operator
	\[
	Z_k : \overline{\Delta_{\Theta} L^2(\mathcal{E}_1)}
	\longrightarrow
	\overline{\Delta_k L^2(\mathcal{E}_k)}
	\oplus \cdots \oplus
	\overline{\Delta_1 L^2(\mathcal{E}_1)}
	\]
	obtained as the unique continuous extension of the densely defined isometry
	\begin{align}\label{k_regular}
		Z_k(\Delta_{\Theta} f)
		\coloneqq
		\Delta_k \Theta_{k-1}\cdots\Theta_1 f
		\oplus \cdots \oplus
		\Delta_2 \Theta_1 f
		\oplus
		\Delta_1 f,
		\quad f\in L^2(\mathcal{E}_1),
	\end{align}
	is unitary, where the associated defect operators are defined as $\Delta_{\Theta} = (I - \Theta^* \Theta)^{1/2}$ and $\Delta_i = (I - \Theta_i^* \Theta_i)^{1/2}$ for $i = 1, \dots, k$. The factorization \eqref{factorization_theta} is said to be a \emph{$k$-strange} factorization if it is not $k$-regular, but there exists a $k$-regular factorization 
	\begin{align*}
		\Theta(z)=\Theta'_k(z)\cdots\Theta'_1(z), \quad \text{for all } z \in \mathbb{D},
	\end{align*}
	such that, for each $i=1,\dots,k$, the purely contractive part of \(\{\mathcal{E}_i, \mathcal{E}_{i+1}, \Theta_i(z)\}\) coincides with the purely contractive part of \(\{\mathcal{E}'_i, \mathcal{E}'_{i+1}, \Theta'_i(z)\}\).

	\section{Regularity of the Factorization of Characteristic Functions of Contractions in Block Upper Triangular Form}

	In the article \cite{F73}, C. Foia\c{s} proved that the characteristic function of a c.n.u.\ contraction $T$ does not admit any strange ($2$-strange) factorization if $\Delta_{\Theta_T}(t)= \left(I - \Theta_{T}(e^{it})^{*}\Theta_{T}(e^{it})\right)^{1/2}$ has finite rank for a.e.\  $t\in[0,2\pi]$. In the following proposition, we generalize this result to the case $k>2$.
	
	\begin{proposition}\label{f_d_prop}
		Let $T\in B(\mathcal{H})$ be a c.n.u.\ contraction, and let $\Theta_T$ denote its characteristic function. Suppose that the rank of the associated defect function $\Delta_{\Theta_T}(t)$ is finite for a.e.\  $t\in[0,2\pi]$. Then the characteristic function $\Theta_{T}$ does not admit any $k$-strange factorization.
	\end{proposition}
	\begin{proof}
		Suppose there exists a $k$-strange factorization 
		\begin{align}\label{theta_k_factor}
			\Theta_{T}(z)=\Theta_k(z)\cdots\Theta_1(z), \quad \text{for all } z \in \mathbb{D}.
		\end{align}
		Then, there exists a $k$-regular factorization 
		\begin{align}\label{Theta'}
			\Theta_{T}(z)=\Theta'_k(z)\cdots\Theta'_1(z), \quad \text{for all } z \in \mathbb{D},
		\end{align}
		such that, for each $i=1,\dots,k$, the purely contractive part of \(\{\mathcal{E}_i, \mathcal{E}_{i+1}, \Theta_i(z)\}\) coincides with the purely contractive part of \(\{\mathcal{E}'_i, \mathcal{E}'_{i+1}, \Theta'_i(z)\}\). Consequently, 
		\begin{align}\label{dimensional_equal}
			\dim(\overline{\im(\Delta_i(t))}) = \dim(\overline{\im(\Delta'_i(t))}) .
		\end{align}
		Because the factorization \eqref{Theta'} is $k$-regular, Proposition 2.13 of \cite{HM26a} implies that the corresponding boundary factorization
		\begin{align*}
			\Theta(e^{it})=\Theta'_k(e^{it})\cdots\Theta'_1(e^{it}),
		\end{align*}
		is a $k$-regular factorization for a.e.\  $t\in[0,2\pi]$. Furthermore, by Proposition 2.7(iv) of \cite{HM26a}, we have
		\[
		\dim(\overline{\im(\Delta_{\Theta_T}(t))}) = \dim(\overline{\im(\Delta'_k(t))}) + \cdots + \dim(\overline{\im(\Delta'_1(t))}).
		\]
		Invoking the relation \eqref{dimensional_equal}, we therefore obtain
		\[
		\dim(\overline{\im(\Delta_{\Theta_T}(t))}) = \dim(\overline{\im(\Delta_k(t))}) + \cdots + \dim(\overline{\im(\Delta_1(t))}).
		\]
		Given that $\dim(\overline{\im(\Delta_{\Theta_T}(t))})$ is finite for a.e.  $t\in[0,2\pi]$, an application of Propositions 2.7(iv) and 2.13 from \cite{HM26a} implies that the factorization \eqref{theta_k_factor} must be $k$-regular, which contradicts our assumption.	
	\end{proof}
	
	It was established by B. Sz.-Nagy in \cite{N70} that, for a pure contraction $T$, the characteristic function $\Theta_T$ admits no strange ($2$-strange) factorization. The following proposition provides a generalization of this result for $k>2$.

	\begin{proposition}\label{pure_prop}
		Let $T\in B(\mathcal{H})$ be a pure contraction. Then the characteristic function $\Theta_{T}$ does not admit any $k$-strange factorization.
	\end{proposition}
	\begin{proof}
		Assume there exists a $k$-strange factorization \begin{align}\label{theta_k_factor_}
			\Theta_{T}(z)=\Theta_k(z)\cdots\Theta_1(z), \quad \text{for all } z \in \mathbb{D}.
		\end{align}
		Then, there exists a $k$-regular factorization 
		\begin{align}\label{Theta'_}
			\Theta_{T}(z)=\Theta'_k(z)\cdots\Theta'_1(z), \quad \text{for all } z \in \mathbb{D},
		\end{align}
		such that, for each $i=1,\dots,k$, the purely contractive part of \(\{\mathcal{E}_i, \mathcal{E}_{i+1}, \Theta_i(z)\}\) coincides with the purely contractive part of \(\{\mathcal{E}'_i, \mathcal{E}'_{i+1}, \Theta'_i(z)\}\). Since the factorization \eqref{Theta'_} is $k$-regular, Theorem 3.5 of \cite{HM26a} ensures that the Sz.-Nagy and Foia\c{s} functional model operator $T_{\Theta_{T}}$ associated with $\Theta_{T}$ admits a block upper triangular matrix representation. Moreover, by Corollary~3.6 of \cite{HM26a}, for each $i=1,\dots,k$, the characteristic function of the diagonal block $P_{\mathcal{H}_i}T_{\Theta_{T}}|_{\mathcal{H}_i}$ coincides with the purely contractive part of \(\{\mathcal{E}'_i, \mathcal{E}'_{i+1}, \Theta'_i(z)\}\).
		
		Since $T$ is a pure contraction, its functional model $T_{\Theta_T}$ is also pure. Because the compression of a pure operator remains pure, it follows that for $i=1,\dots,k$, the diagonal block $P_{\mathcal{H}_i}T_{\Theta_{T}}|_{\mathcal{H}_i}$ is a pure contraction. This implies that the characteristic function of $P_{\mathcal{H}_i}T_{\Theta_{T}}|_{\mathcal{H}_i}$ is an inner function. Consequently, the purely contractive part of \(\{\mathcal{E}'_i, \mathcal{E}'_{i+1}, \Theta'_i(z)\}\) is inner, and therefore, the purely contractive part of \(\{\mathcal{E}_i, \mathcal{E}_{i+1}, \Theta_i(z)\}\) must be inner as well. This dictates that the corresponding operator \(\Theta_i \in B\left(L^2(\mathcal{E}_i), L^2(\mathcal{E}_{i+1})\right)\) is an isometry. Finally, invoking Proposition 2.7(ii) of \cite{HM26a}, we conclude that the factorization \eqref{theta_k_factor_} is a $k$-regular factorization, which yields a contradiction.
				\end{proof}

	Let \(T \in B(\mathcal{H})\) be a contraction with a closed invariant subspace \(\mathcal{H}_1\). Under these conditions, Sz.-Nagy and Foia\c{s} established a canonical factorization for the characteristic function of $T$, as detailed in the following theorem. 
	
	\begin{theorem}[\cite{NF67}, Theorem 1] \label{trm:1}
		Let \(\mathcal{H}_1\) and \(\mathcal{H}_2\) be Hilbert spaces, and let
		\[
		T =
		\begin{bmatrix}
			A & X \\
			0 & B
		\end{bmatrix} :\mathcal H_1 \oplus\mathcal H_2 \to\mathcal H_1 \oplus\mathcal H_2
		\]
		be a contraction on \(\mathcal{H}_1 \oplus \mathcal{H}_2\). Then there exists a contraction \(L \in B(\mathcal{D}_{B}, \mathcal{D}_{A^*})\) and canonical unitary operators \(\tau \in B(\mathcal{D}_T, \mathcal{D}_{A} \oplus \mathcal{D}_L)\) and \(\tau_* \in B(\mathcal{D}_{T^*}, \mathcal{D}_{B^*} \oplus \mathcal{D}_{L^*})\) such that \(X = D_{A^*} L D_{B}\), and
		\begin{equation} \label{J_factorization}
			\Theta_T(z) =
			\tau_*^{-1}
			\begin{bmatrix}
				\Theta_{B}(z) & 0 \\
				0 & I_{\mathcal{D}_{L^*}}
			\end{bmatrix}
			\begin{bmatrix}
				L^* & D_L \\
				D_{L^*} & -L
			\end{bmatrix}
			\begin{bmatrix}
				\Theta_{A}(z) & 0 \\
				0 & I_{\mathcal{D}_L}
			\end{bmatrix}
			\tau~ \text{ for all }z \in \mathbb{D}.
		\end{equation}
	\end{theorem}

	With the aid of this theorem, we establish the following result. 
	\begin{proposition}\label{prop:pure-block}
		Suppose $\mathcal{H}_1$ and $\mathcal{H}_2$ are Hilbert spaces, and let
		\[
		T=\begin{bmatrix} A & D_{A^*} L D_B \\[4pt] 0 & B \end{bmatrix} :\mathcal H_1 \oplus\mathcal H_2 \to\mathcal H_1 \oplus\mathcal H_2
		\]
		be a c.n.u.\ contraction. Then \(T\) is a pure contraction if and only if both \(A\) and \(B\) are pure contractions. 
	\end{proposition}
	
	\begin{proof}
		Observe that 
		\begin{align}
			A^{*n} = P_{\mathcal H_1} T^{*n}|_{\mathcal H_1} 
			\quad \text{and} \quad 
			B^{*n} = P_{\mathcal H_2} T^{*n}|_{\mathcal H_2}, \quad n \geq 1.
		\end{align}
		Hence, if \(T\) is pure, it follows immediately that both \(A\) and \(B\) are pure contractions. 
		
		Conversely, suppose that \(A\) and \(B\) are pure contractions. By Proposition~3.5 of Chapter~VI in \cite{NFBK10}, a c.n.u.\ contraction is pure if and only if its characteristic function is inner. Since \(\Theta_A\) and \(\Theta_B\) are inner, it follows from the factorization \eqref{J_factorization} that \(\Theta_T\) is also inner. Therefore, \(T\) is pure.
	\end{proof}
	%%%%%%%%%%%%%%%%%%%%%%%%%%%%%%%%%%%%%%%%%%%%
	
	In the Theorem \ref{trm:1}, let us introduce the following notation:
	\[
	\Theta'_3(z) := \tau_*^{-1}
	\begin{bmatrix}
		\Theta_{B}(z) & 0 \\
		0 & I_{\mathcal{D}_{L^*}}
	\end{bmatrix},~	\Theta'_2(z) := 
	\begin{bmatrix}
		L^* & D_L \\
		D_{L^*} & -L
	\end{bmatrix} ~
	\Theta'_1(z) := 
	\begin{bmatrix}
		\Theta_{A}(z) & 0 \\
		0 & I_{\mathcal{D}_L}
	\end{bmatrix} \tau, \quad z \in \mathbb{D}.
	\]
	With this notation, the characteristic function admits the factorization
	\begin{equation} \label{fact_I}
		\Theta_T(z) = \Theta'_3(z)\Theta'_2(z)\Theta'_1(z), \quad z \in \mathbb{D}.
	\end{equation}
	Our objective is to determine whether the factorization \eqref{fact_I} is $3$-regular.
	
	\begin{theorem} \label{theorem_either}
		Let 
		\(
		T = \begin{bmatrix} A & D_{A^*} L D_{B} \\ 0 & B \end{bmatrix} :\mathcal{H}_1 \oplus\mathcal{H}_2 \to\mathcal{H}_1 \oplus\mathcal{H}_2
		\) 
		be a c.n.u.\ contraction. Then the factorization of the characteristic function \(\Theta_T\) given in \eqref{fact_I} is either a $3$-regular factorization or a $3$-strange factorization.
	\end{theorem}
	
	\begin{proof}
		Since $\mathcal{H}_1$ is a non-trivial invariant subspace for the c.n.u.\ contraction $T$, Theorem 1.1 and Proposition 2.1 of Chapter VII in \cite{NFBK10} guarantee the existence of a regular factorization 
		\begin{align}\label{fact_r}
			\Theta_T(z) = \Theta_2(z)\,\Theta_1(z), \quad z \in \mathbb{D},
		\end{align}
		such that the purely contractive parts of $\{\mathcal{D}_T,\mathcal{D},\Theta_1(z)\}$ and $\{\mathcal{D},\mathcal{D}_{T^*},\Theta_2(z)\}$  coincide with the characteristic functions $\{\mathcal{D}_A,\mathcal{D}_{A*},\Theta_{A}(z)\}$ and $\{\mathcal{D}_B,\mathcal{D}_{B*},\Theta_{B}(z)\}$, respectively. Let us now consider the factorization
		\begin{align}\label{fact_"}
			\Theta_T(z) = \Phi_3(z)\Phi_2(z)\Phi_1(z), \quad z \in \mathbb{D},
		\end{align}
		where 
		\[ \Phi_3(z)=\Theta_2(z),~\Phi_2(z)=I_{\mathcal{D}},~\Phi_1(z)=\Theta_1(z).\]
		Since the factorization \eqref{fact_r} is $2$-regular and $\Phi_2(z)=I_{\mathcal{D}}$, it immediately follows that the factorization \eqref{fact_"} is a $3$-regular factorization. Moreover, it is evident that the purely contractive parts of $\Phi_3(z),~\Phi_2(z),\text{ and }\Phi_1(z)$ coincide with $\Theta'_3(z),~\Theta'_2(z),\text{ and }\Theta'_1(z)$, respectively. Therefore, the result follows.
	\end{proof}

	\begin{corollary}\label{coro:1}
		Let 
		\(
		T = \begin{bmatrix} A & D_{A^*} L D_{B} \\ 0 & B \end{bmatrix} :\mathcal{H}_1 \oplus\mathcal{H}_2 \to\mathcal{H}_1 \oplus\mathcal{H}_2
		\) be a c.n.u.\ contraction such that \(\dim(\Delta_{\Theta_T}(t))<\infty\), for a.e.\  $t\in[0,2\pi].$ Then the factorization of the characteristic function \(\Theta_T\) given in equation {\rm \eqref{fact_I}} is a $3$-regular factorization.
	\end{corollary}

	\begin{proof}
		By Proposition~\ref{f_d_prop}, the factorization \eqref{fact_I} cannot be a $3$-strange factorization. Therefore, it follows from Theorem~\ref{theorem_either} that the factorization \eqref{fact_I} is a $3$-regular factorization.
	\end{proof}

	\begin{remark}\label{remark_1}
		Based on the definition of $3$-regularity provided in \eqref{k_regular}, it is evident that the factorization \eqref{fact_I} is $3$-regular if and only if the factorization 
		\begin{align}\label{main_fact}
			\Theta(z)=\Theta_2(z)\Theta_1(z), \quad \text{for all } z \in \mathbb{D},
		\end{align}
		where $\Theta_1(z)= \begin{bmatrix}
			L^* & D_L \\
			D_{L^*} & -L
		\end{bmatrix}
		\begin{bmatrix}
			\Theta_{A}(z) & 0 \\
			0 & I_{\mathcal{D}_L}
		\end{bmatrix}$ and $\Theta_2(z)=\begin{bmatrix}
			\Theta_{B}(z) & 0 \\
			0 & I_{\mathcal{D}_{L^*}}
		\end{bmatrix}$, is a $2$-regular factorization.
	\end{remark}
	
	\begin{lemma}\label{lemma_main}
		The factorization \(\Theta_T(z) = \Theta'_3(z)\Theta'_2(z)\Theta'_1(z)\) given in \eqref{fact_I} is $3$-regular if and only if
		\[
		\Delta_{B}(t)\mathcal{D}_{B} \cap L^*\left[\Delta_{*A}(t)\mathcal{D}_{A^*} \cap \ker(D_{L^*})\right] = \{0\}, \quad \text{for a.e.\ } t\in[0,2\pi],
		\]
		where $\Delta_{B}(t)=\left(I_{\mathcal{D}_B}-\Theta_B^*(e^{it})\Theta_B(e^{it})\right)^{1/2}$ and $\Delta_{*A}(t)=\left(I_{\mathcal{D}_{A^*}}-\Theta_A(e^{it})\Theta_A^*(e^{it})\right)^{1/2}$.
	\end{lemma}
	
	\begin{proof}
		By Proposition 3.1 in Chapter VII of \cite{NFBK10} and the proposition in \cite{NF74}, the factorization \( \Theta(z) = \Theta_2(z)\Theta_1(z) \) defined in \eqref{main_fact} is $2$-regular if and only if
		\[
		\Delta_2(t)(\mathcal{D}_{B} \oplus \mathcal{D}_{L^*}) \cap \Delta_{*1}(t)(\mathcal{D}_{B} \oplus \mathcal{D}_{L^*}) = \{0\} \quad \text{for a.e.\ } t\in[0,2\pi].
		\]
		We compute:
		\begin{align*}
			\Delta_2(t)&(\mathcal{D}_{B} \oplus \mathcal{D}_{L^*}) \cap \Delta_{*1}(t)(\mathcal{D}_{B} \oplus \mathcal{D}_{L^*}) \\
			=&\left\{(I - \Theta_2^*(e^{it})\Theta_2(e^{it}))^{1/2}(\mathcal{D}_{B} \oplus \mathcal{D}_{L^*})\right\}\cap \left\{(I - \Theta_1(e^{it})\Theta_1^*(e^{it}))^{1/2}(\mathcal{D}_{B} \oplus \mathcal{D}_{L^*})\right\} \\
			=&\left\{
			\begin{bmatrix}
				\Delta_{B}(t) & 0\\
				0 & 0
			\end{bmatrix}
			\begin{bmatrix}
				\mathcal{D}_{B} \\
				\mathcal{D}_{L^*}
			\end{bmatrix}
			\right\}
			\cap
			\left\{
			\begin{bmatrix} 
				L^* & D_{L} \\
				D_{L^*} & -L
			\end{bmatrix}
			\begin{bmatrix}
				\Delta_{*A}(t) & 0\\
				0 & 0
			\end{bmatrix}
			\begin{bmatrix} 
				L & D_{L^*} \\
				D_{L} & -L^*
			\end{bmatrix}
			\begin{bmatrix}
				\mathcal{D}_{B} \\
				\mathcal{D}_{L^*}
			\end{bmatrix}
			\right\} \\
			=&\left\{
			\begin{bmatrix}
				\Delta_{B}(t) & 0\\
				0 & 0
			\end{bmatrix}
			\begin{bmatrix}
				\mathcal{D}_{B} \\
				\mathcal{D}_{L^*}
			\end{bmatrix}
			\right\}
			\cap
			\left\{
			\begin{bmatrix} 
				L^* & D_{L} \\
				D_{L^*} & -L
			\end{bmatrix}
			\begin{bmatrix}
				\Delta_{*A}(t) & 0\\
				0 & 0
			\end{bmatrix}
			\begin{bmatrix}
				\mathcal{D}_{A^*} \\
				\mathcal{D}_{L}
			\end{bmatrix}
			\right\} \\
			=& \left\{ \Delta_{B}(t)\mathcal{D}_{B} \oplus \{0\} \right\}
			\cap
			\left\{
			\begin{bmatrix} 
				L^* & D_{L} \\
				D_{L^*} & -L
			\end{bmatrix}
			\begin{bmatrix}
				\Delta_{*A}(t)\mathcal{D}_{A^*} \\
				\{0\}
			\end{bmatrix}
			\right\} \\
			=& \left\{ \Delta_{B}(t)\mathcal{D}_{B} \oplus \{0\} \right\}
			\cap
			\left\{ L^*x \oplus D_{L^*}x : x \in \Delta_{*A}(t)\mathcal{D}_{A^*} \right\} \\
			=& \left\{ \Delta_{B}(t)\mathcal{D}_{B} \cap L^*\left[ \Delta_{*A}(t)\mathcal{D}_{A^*} \cap \ker(D_{L^*}) \right] \right\} \oplus \{0\}.
		\end{align*}
		This equivalence, combined with Remark~\ref{remark_1}, yields the desired conclusion.
	\end{proof}

	\begin{theorem}\label{several_classes}
		Let \( T =
		\begin{bmatrix}
			A & X \\
			0 & B
		\end{bmatrix} \)
		be a contraction on \(\mathcal{H}_1 \oplus \mathcal{H}_2\). Then, in the following cases, the factorization in equation~{\rm\eqref{fact_I}} of the characteristic function \(\Theta_T\) is a $3$-regular factorization
		\begin{enumerate}[\rm (i)]
			\item If the operator \(A^*\) or \(B\) is a pure contraction.
			\item If the operator \(T\) is a pure contraction.
			\item If the coupling operator \(L\) is a proper contraction, i.e., \(\|L^*x\| < \|x\|\) for all nonzero \(x \in \mathcal{D}_{A^*}\), or \(\|Ly\| < \|y\|\) for all nonzero \(y \in \mathcal{D}_{B}\).
		\end{enumerate}
	\end{theorem}
	
	\begin{proof}
		
		We establish the assertions for each case as follows:
		
		\noindent(i) If \(A^*\) is a pure contraction, then \(\Theta_{A^*}(e^{it})\) is an isometry a.e.\  $t\in[0,2\pi].$ 
			Since \(\Theta_{A^*}(e^{it}) = \Theta_{A}(e^{-it})^*\), it follows that \(\Theta_{A}(e^{it})^*\) is also an isometry a.e.\  $t\in[0,2\pi].$ Consequently, \(\Delta_{*1}(t) = 0\) a.e.\  $t\in[0,2\pi].$
			By Lemma~\ref{lemma_main}, this ensures that the factorization of the characteristic function \(\Theta_T\) given in equation~{\rm\eqref{fact_I}} is $3$-regular. 
			Similarly, if \(B\) is a pure contraction, then \(\Theta_{B}\) is an inner function, which implies that \(\Delta_{B}(t)\mathcal{D}_{B} = \{0\}\) a.e.\  $t\in[0,2\pi].$
			Consequently, invoking Lemma~\ref{lemma_main} again, the factorization of the characteristic function \(\Theta_T\) in equation~{\rm\eqref{fact_I}} is $3$-regular.
			
			\noindent(ii) If \(T\) is a pure contraction, then \(B\) must necessarily be a pure contraction. 
			Hence, the desired result follows directly from Case~(i).
			
			\noindent(iii) In this case, $\ker(D_{L^*})=\{0\}$. By Lemma~\ref{lemma_main}, the result follows immediately. \qedhere
	\end{proof}
	
	\begin{remark}
		The case in Theorem~\ref{several_classes} where $T$ is pure may alternatively be established via Propositions~\ref{pure_prop} and \ref{theorem_either}.
	\end{remark}
	
	\section{Counter Examples}
	
	Although the factorization of the characteristic function of a contraction in block upper triangular form given in \eqref{fact_I} is $3$-regular for several classes of contractions, as discussed in the previous section, this property does not hold universally. The following counterexample demonstrates a case where this factorization fails to be $3$-regular.

	\begin{example}\label{exp:1}
		Let \(T: H^2(\mathbb{D}) \oplus H^2(\mathbb{D}) \to H^2(\mathbb{D}) \oplus H^2(\mathbb{D})\) be the contraction defined by
		\[
		T := \begin{bmatrix} S & D_{S^*} L D_{S^*} \\ 0 & S^* \end{bmatrix},
		\]
		where \(S\) denotes the unilateral shift on the Hardy space \(H^2(\mathbb{D})\) and \(L = I_{\mathcal{D}_{S^*}}\). One observes that \(D_S = 0\) and \(D_{S^*} = P_{\mathbb{C}}\). Consequently, the characteristic function of \(S\) is given by  
		\[
		\Theta_S : \mathbb{D} \to B(\{0\}, \mathbb{C}), \quad \Theta_S(z) = 0_{\{0\}\to \mathbb{C}} \quad \text{for all } z \in \mathbb{D}.
		\]
		Similarly, the characteristic function of \(S^*\) takes the form  
		\[
		\Theta_{S^*} : \mathbb{D} \to B(\mathbb{C}, \{0\}), \quad \Theta_{S^*}(z) = 0_{\mathbb{C}\to\{0\}} \quad \text{for all } z \in \mathbb{D},
		\]
		where \(0_{\mathcal{E} \to \mathcal{F}}\) denotes the zero operator from \(\mathcal{E}\) to \(\mathcal{F}\). To apply Lemma~\ref{lemma_main}, observe that $A=S$, $B=S^*$, and $L=I_{\mathcal{D}_{S^*}}$. Evaluating the relevant intersection, we obtain
		\begin{align*}
			\Delta_{B}(t)&\mathcal{D}_{B} \cap L^*\left[\Delta_{*A}(t)\mathcal{D}_{A^*} \cap \ker(D_{L^*})\right]\\
			& =I_{\mathbb{C}}(\mathbb{C})\cap I_{\mathbb{C}}[I_{\mathbb{C}}(\mathbb{C})\cap\mathbb{C}]\\
			&=\mathbb{C}\neq \{0\} \quad \text{for a.e. } t\in[0,2\pi].
		\end{align*}
		Consequently, Lemma~\ref{lemma_main} implies that the factorization \eqref{fact_I} is \emph{not} $3$-regular.
	\end{example}
	
	The operator $T$ in Example~\ref{exp:1} is unitary. To show that $3$-regularity can also fail for non-unitary contractions, we construct the following example.

	\begin{example}\label{exp:2}
		Let $T : \displaystyle\bigoplus_{i=1}^{4} H^2(\mathbb{D}) \to \bigoplus_{i=1}^{4} H^2(\mathbb{D})$ be the contraction defined by
		\[
		T := \begin{bmatrix}
			S^* \oplus S & D_{S \oplus S^*} L D_{S \oplus S^*} \\
			0 & S \oplus S^*
		\end{bmatrix},
		\]
		where $S$ is the shift operator on the Hardy space $H^2(\mathbb{D})$, and $L = I_{\mathcal{D}_{S \oplus S^*}}$. A direct computation of the defect operators yields
		\[
		D_{S^* \oplus S} = P_{\mathbb{C}} \oplus 0, \quad D_{S \oplus S^*} = 0 \oplus P_{\mathbb{C}}.
		\]
		Thus, the operator $T$ takes the form
		\[
		T = \begin{bmatrix}
			S^* \oplus S & 0 \oplus P_{\mathbb{C}} \\
			0\oplus 0 & S \oplus S^*
		\end{bmatrix}.
		\]
		Evaluating $T T^*$ and $T^* T$, we obtain:
		\begin{align*}
			T T^* &= 
			\begin{bmatrix}
				S^* \oplus S & 0 \oplus P_{\mathbb{C}} \\
				0\oplus 0 & S \oplus S^*
			\end{bmatrix}
			\begin{bmatrix}
				S \oplus S^* & 0\oplus 0 \\
				0 \oplus P_{\mathbb{C}} & S^* \oplus S
			\end{bmatrix} 
			= 
			\begin{bmatrix}
				I \oplus I & 0\oplus 0\\
				0\oplus 0 & S^* S \oplus I
			\end{bmatrix}, \\
			T^* T &= 
			\begin{bmatrix}
				S \oplus S^* & 0\oplus 0 \\
				0 \oplus P_{\mathbb{C}} & S^* \oplus S
			\end{bmatrix}
			\begin{bmatrix}
				S^* \oplus S & 0 \oplus P_{\mathbb{C}} \\
				0\oplus 0 & S \oplus S^*
			\end{bmatrix} 
			= 
			\begin{bmatrix}
				S^* S \oplus I & 0\oplus 0 \\
				0\oplus 0 & I \oplus I
			\end{bmatrix}.
		\end{align*}
		Consequently, $T$ is not a unitary operator, despite its restriction to the subspace $\{0\} \oplus H^2(\mathbb{D}) \oplus \{0\} \oplus H^2(\mathbb{D})$ being unitary.
		The contractive analytic functions
		$\{\mathbb{C}\oplus \{0\},\{0\}\oplus\mathbb{C},0\} $ and $\{\{0\}\oplus\mathbb{C},\mathbb{C}\oplus \{0\},0\} $ represent the characteristic functions of $S^*\oplus S$ and $S\oplus S^*$, respectively.
		
		Applying Lemma~\ref{lemma_main} and noting that  \(A=S^*\oplus S\), \(B=S\oplus S^*\), and \(L=I_{\mathcal{D}_{S \oplus S^*}}\), we obtain
		\begin{align*}
			\Delta_{*A}(t)=\Delta_{B}(t)=I_{\{0\}\oplus\mathbb{C}}, \quad \text{for a.e. } t\in[0,2\pi].
		\end{align*}
		Consequently, evaluating the relevant intersection yields
		\begin{align*}
			\Delta_{B}(t)&\mathcal{D}_{B} \cap L^*\left[\Delta_{*A}(t)\mathcal{D}_{A^*} \cap \ker(D_{L^*})\right]\\
			& =(\{0\}\oplus\mathbb{C})\cap (\{0\}\oplus\mathbb{C})\\
			&=\{0\}\oplus\mathbb{C}\neq \{0\} \quad \text{for a.e. } t\in[0,2\pi].
		\end{align*}
		Therefore, Lemma~\ref{lemma_main} implies that the factorization of the characteristic function given in \eqref{fact_I} is not $3$-regular.
	\end{example}
	Finally, we construct another example utilizing the functional model operator associated with the contractive analytic function $\{\mathbb{C}, \mathbb{C}, -\tfrac{1}{\sqrt{2}}\}$ considered in \cite{F73}, with a direct sum of unilateral shift operators.

	%%%%%%%%%%%%%%
	
	\begin{example} 
		Let \(\Theta\) denote the contractive analytic function given by \(\{\mathbb{C}, \mathbb{C}, -\frac{1}{\sqrt{2}}\}\). Let \(T_\Theta\) be the corresponding model operator, defined on the model space
		\[
		\mathcal{H}_{\Theta} := H^2(\mathbb{D}) \oplus L^2(\mathbb{C}) \ominus \left\{ -\frac{1}{\sqrt{2}} w \oplus \frac{1}{\sqrt{2}} w \mid w \in H^2(\mathbb{D}) \right\},
		\]
		by its action on the elements
		\[
		T_{\Theta}(e_n) = \begin{cases}
			e_{n+1} & \text{if } n \neq -1, \\
			\frac{1}{\sqrt{2}} e_0 & \text{if } n = -1,
		\end{cases}
		\]
		where the sequence
		\[
		e_n = \begin{cases}
			0 \oplus e^{int} & \text{if } n < 0, \\
			\frac{1}{\sqrt{2}} e^{int} \oplus \frac{1}{\sqrt{2}} e^{int} & \text{if } n \geq 0,
		\end{cases}
		\]
		constitutes an orthonormal basis for \(\mathcal{H}_\Theta\). Define the contraction \(T\) on the Hilbert space \(\bigoplus_{i=0}^{\infty} H^2(\mathbb{D}) \oplus \mathcal{H}_\Theta\) by the block matrix
		\[
		T = \begin{bmatrix}
			A & X \\
			0 & B
		\end{bmatrix},
		\]
		where \(A := \bigoplus_{i=0}^{\infty} S\). For this operator, we have
		\[
		\Theta_{A}(z) = 0, \quad D_{A} = 0, \quad D_{A^*} = \bigoplus_{i=0}^{\infty} P_{\mathbb{C}}, \quad \mathcal{D}_{A} = \{0\}, \quad \mathcal{D}_{A^*} = \bigoplus_{i=0}^{\infty} \mathbb{C},
		\]
		where \(P_{\mathbb{C}}\) denotes the orthogonal projection onto the subspace of constant functions in \(H^2(\mathbb{D})\). Furthermore, we set \(B := T_\Theta\), which yields
		\[
		\Theta_{B}(z) = -\frac{1}{2}, \quad D_{B} = \text{diag}(\cdots, 0, \frac{1}{\sqrt{2}}, \underline{0}, 0, \cdots), \quad D_{B^*} = \text{diag}(\cdots, 0, \underline{\frac{1}{\sqrt{2}}}, 0, 0, \cdots),
		\]
		and
		\[
		\mathcal{D}_{B} = \cdots \oplus \{0\} \oplus \mathbb{C} \oplus \underline{\{0\}} \oplus 0 \oplus \cdots, \quad \mathcal{D}_{B^*} = \cdots \oplus \{0\} \oplus \underline{\mathbb{C}} \oplus \{0\} \oplus 0 \oplus \cdots.
		\]
		Next, define \(X = D_{A^*} L D_{B}\), where the contraction \(L: \mathcal{D}_{B} \to \mathcal{D}_{A^*}\) is specified by
		\[
		L(\cdots, 0, x, \underline{0}, \cdots) := (x, 0, 0, \cdots).
		\]
		A direct calculation reveals that
		\[
		\Delta^2_{*A}(t) = I_{\mathcal{D}_{A^*}} - \Theta_{A}(e^{it}) \Theta_{A}^*(e^{it}) = I_{\mathcal{D}_{A^*}}, \quad \Delta_{*A}(t) \mathcal{D}_{A^*} = \mathcal{D}_{A^*},
		\]
		and
		\[
		\Delta^2_{B}(t) = I_{\mathcal{D}_{B}} - \Theta_{B}^*(e^{it}) \Theta_{B}(e^{it}) = \frac{1}{2} I_{\mathcal{D}_{B}}, \quad \Delta_{B}(t) \mathcal{D}_{B} = \mathcal{D}_{B}.
		\]	
		Thus, the corresponding intersection is
		\begin{align*}
			\Delta_{B}(t)&\mathcal{D}_{B} \cap L^*\left[\Delta_{*A}(t)\mathcal{D}_{A^*} \cap \ker(D_{L^*})\right]\\
			&= \mathcal{D}_{B} \cap L^*(\ker(D_{L^*}))\\
			& =\mathcal{D}_{B} \cap L^* (L(\mathcal{D}_{B}))\\
			&=\mathcal{D}_{B}\neq \{0\} \quad \text{for a.e. } t\in[0,2\pi].
		\end{align*}

		Hence, Lemma~\ref{lemma_main} implies that the factorization of the characteristic function given in \eqref{fact_I} is \emph{not} $3$-regular. Furthermore, one can verify that \(T\) is an isometry; indeed, its defect space $\mathcal{D}_T$ is unitarily equivalent to $\mathcal{D}_A\oplus\mathcal{D}_L=\{0\}$. Since every c.n.u.\ isometry is a pure operator, and the factorization \eqref{fact_I} is always $3$-regular for pure operators, it follows that \(T\) is \emph{not} a c.n.u.\ contraction.
	\end{example}
	In summary, the three counterexamples above demonstrate that the factorization \eqref{fact_I} is not $3$-regular in general. Notably, in each instance, the underlying contraction fails to be c.n.u.\ contraction. Consequently, whether this factorization is necessarily $3$-regular for an arbitrary c.n.u.\ contraction remains an open question.

	\section{Regularity of the Factorization of Characteristic Functions of Row Contractions in Block Upper Triangular Form}

	Let $\mathbb{N}$ and $\mathbb{C}$ denote the sets of natural and complex numbers, respectively. For $n \in \mathbb{N}$, let $\mathbb{C}^n$ denote the standard $n$-dimensional complex Hilbert space, equipped with the canonical orthonormal basis $\{e_1, \ldots, e_n\}$. In the infinite-dimensional case where $n = \infty$, the notation $\mathbb{C}^n$ is conventionally understood to represent the separable Hilbert space $\ell^2(\mathbb{N})$ of all square-summable complex sequences.
	\begin{definition}
		Let \(\mathcal{H}\) be a Hilbert space, and let \(T_1, \dots, T_n \in B(\mathcal{H})\). The \(n\)-tuple \(T := [T_1, \dots, T_n]\), viewed as a bounded linear operator from \(\bigoplus_{i=1}^n \mathcal{H}\) to \(\mathcal{H}\), is termed a \emph{row contraction} if it satisfies the operator inequality
		\[
		\sum_{i=1}^n T_i T_i^* \leq I_{\mathcal{H}},
		\]
		where \(I_{\mathcal{H}}\) denotes the identity operator on \(\mathcal{H}\). For such a row contraction \(T\), the associated defect operators are defined as
		\[
		D_T := (I - T^*T)^{1/2} \in B\left(\bigoplus_{i=1}^n \mathcal{H} \right), \quad \text{and} \quad D_{T^*} := (I - TT^*)^{1/2} \in B(\mathcal{H}).
		\]
		The corresponding defect spaces are defined as the closures of the ranges of these operators, namely
		\[
		\mathcal{D}_T := \overline{\im(D_T)}, \quad \mathcal{D}_{T^*} := \overline{\im(D_{T^*})},
		\]
		and their respective dimensions are referred to as the \emph{defect indices} of \(T\).
	\end{definition}

	Let $\Gamma(\mathbb{C}^n)$ denote the \emph{full Fock space} over $\mathbb{C}^n$, defined by
	
	\[
	\Gamma(\mathbb{C}^n) := \bigoplus_{k \geq 0} (\mathbb{C}^n)^{\otimes k} = \mathbb{C} \oplus \mathbb{C}^n \oplus (\mathbb{C}^n \otimes \mathbb{C}^n) \oplus  \cdots\oplus (\mathbb{C}^n)^{\otimes^k} \oplus\cdots.
	\]
	For notational convenience, we write $\Gamma$ in place of $\Gamma(\mathbb{C}^n)$. The \textit{vacuum vector} $1 \oplus 0 \oplus \cdots \in \Gamma$ is denoted by $e_\emptyset$. Let $\mathbb{F}_n^+$ denote the unital free semigroup on $n$ generators $1, \dots, n$, with the empty word $\emptyset$ serving as its identity element. Given a word $\alpha = \alpha_1 \cdots \alpha_j \in \mathbb{F}_n^+$, we denote the elementary tensor $e_{\alpha_1} \otimes \cdots \otimes e_{\alpha_j}$ by $e_\alpha$. The set $\{e_\alpha : \alpha \in \mathbb{F}_n^+\}$ then constitutes an orthonormal basis for $\Gamma$. The length of a word $\alpha \in \mathbb{F}_n^+$ is defined as $|\alpha| = j$ if $\alpha = \alpha_1 \dots \alpha_j$, and $|\alpha| = 0$ if $\alpha=\emptyset$. 	Let $T = [T_1, \cdots, T_n]$ be a row contraction on a Hilbert space $\mathcal{H}$. For any word $\alpha = \alpha_1 \cdots \alpha_j \in \mathbb{F}_n^+$, we define the product operator $T_\alpha := T_{\alpha_1} \cdots T_{\alpha_j}$, with the convention that $T_\emptyset := I_{\mathcal{H}}$. Furthermore, for each $i = 1, \dots, n$, the \emph{left creation operator} $S_i : \Gamma \to \Gamma$ is defined by
	\[
	S_i(x) := e_i \otimes x \quad \text{for all } x \in \Gamma.
	\]
	Similarly, the \emph{right creation operator} $R_i: \Gamma \to \Gamma$ is defined by
	\[
	R_i(x) := x \otimes e_i \quad \text{for all }  x \in \Gamma.
	\]
	
	\begin{definition}
		Let \(\mathcal{H}\) be a Hilbert space and consider a row contraction \(T = [T_1, \ldots, T_n]: \bigoplus_{i=1}^n \mathcal{H} \to \mathcal{H}\). Two fundamental subclasses of row contractions are defined as follows:
		\begin{itemize}
			\item \emph{Completely non-coisometric}: Let \(\mathcal{H}_c\) denote the subspace of \(\mathcal{H}\) given by
			\[
			\mathcal{H}_c := \left\{ h \in \mathcal{H} : \sum_{|\alpha| = k} \|T_{\alpha}^* h\|^2 = \|h\|^2 \text{ for all } k \geq 1 \right\}.
			\]
			The row contraction \(T\) is said to be \emph{completely non-coisometric} (c.n.c.)\ provided that \(\mathcal{H}_c = \{0\}\).
			
			\item \emph{Pure}: The row contraction \(T\) is termed \emph{pure} if, for every vector \(h \in \mathcal{H}\), the condition
			\[
			\lim_{k \to \infty} \sum_{|\alpha| = k} \|T_{\alpha}^* h\|^2 = 0
			\]
			is satisfied.
		\end{itemize}
	\end{definition}
	
	Let $\mathcal{H}$ and $\mathcal{K}$ be Hilbert spaces, and let $\Gamma$ denote the full Fock space over $\mathbb C^n$. A bounded linear operator $M :\Gamma\otimes\mathcal{H}\to\Gamma\otimes\mathcal{K}$ is said to be a \emph{multi-analytic operator} if it intertwines the left creation operators, that is,
	$$M(S_i\otimes I_{\mathcal{H}})=(S_i\otimes I_\mathcal{K})M\quad \text{for } i=1,\dots,n.$$
	
	Let $N : \Gamma\otimes\mathcal{H}'\to\Gamma\otimes\mathcal{K}'$ be another multi-analytic operator. The operator $M$ is said to \emph{coincide} with $N$ if there exist unitary operators $U\in B(\mathcal{H},\mathcal{H}')$ and $V\in B(\mathcal{K},\mathcal{K}')$ such that
	$$N(I_\Gamma\otimes U)=(I_\Gamma \otimes V)M.$$
	
	For a given multi-analytic operator \( M: \Gamma \otimes \mathcal{H} \to \Gamma \otimes \mathcal{K} \), define the operator $\theta : \mathcal{H}\to \Gamma\otimes\mathcal{K}$ by
	\[
	\theta h=M(e_\emptyset\otimes h), \quad h\in\mathcal{H}.
	\]
	It is well known that the operator \( M \) is uniquely determined by \( \theta \). The operator $\theta$ is called the \emph{symbol} of \( M \), and we denote the corresponding multi-analytic operator by \( M_\theta \).

	A multi-analytic operator $M_\theta : \Gamma \otimes \mathcal{H} \to \Gamma \otimes \mathcal{K}$ is said to be
	\begin{enumerate}[\rm (i)]
		\item \emph{inner} if $M_\theta$ is an isometry;
		\item \emph{outer} if its range is dense, that is,
		\[
		\overline{M_\theta(\Gamma \otimes \mathcal{H})} = \Gamma \otimes \mathcal{K};
		\]
		\item a \emph{unitary constant} if $M_\theta$ takes the form $M_\theta = I_\Gamma \otimes W$, where $W\in B(\mathcal{H}, \mathcal{K})$ is a unitary operator;
		\item \emph{purely contractive} if, for every nonzero $h \in \mathcal{H}$, the following strict inequality holds:
		\[
		\| P_{\mathcal{K}} \theta h \| < \| h \|,
		\]
		where $P_{\mathcal{K}}$ denotes the orthogonal projection from $\Gamma \otimes \mathcal{K}$ onto the subspace $e_{\emptyset} \otimes \mathcal{K}$.
	\end{enumerate}

	The notion of a characteristic function was extended to the setting of row contractions by G.~Popescu in \cite{Po89b}; for a detailed exposition, see \cite{Po89a,Po95,Po06}. For a row contraction $T$ on a Hilbert space $\mathcal{H}$, consider the multi-analytic operator $\Theta_{T} : \Gamma\otimes\mathcal{D}_{T} \to \Gamma\otimes\mathcal{D}_{T^*}$ induced by the symbol $\theta_{T}$, which is explicitly defined for $h \in \mathcal{D}_{T}$ by the formula
	\[
	\theta_{T}(h) \coloneq -\displaystyle\sum_{i=1}^n T_iP_ih + \displaystyle\sum_{i=1}^n  (S_i\otimes I_{\mathcal{D}_{T^*}}) \left(\displaystyle\sum_{\alpha\in \mathbb{F}_n^+} e_{\alpha}\otimes D_{T^*}T^*_{\alpha}P_i D_{T}h\right).
	\]
	Here, $P_i$ denotes the canonical orthogonal projection from $\bigoplus_{i=1}^{n}\mathcal{H}$ onto its $i$-th coordinate space, and $S_i$ is the $i^{\text{th}}$ left creation operator acting on the full Fock space $\Gamma$. The operator $\Theta_T$ is termed the \emph{characteristic function} of the row contraction $T$. Moreover, it is established in Remark~3.2 of \cite{Po89b} that if $T$ is a pure row contraction, the corresponding characteristic function $\Theta_T$ is necessarily inner.
	
	%\section{Regular Factorization of Characteristic Function of Row Contractions}
	Let $\mathcal{E}$ and $\mathcal{E_*}$ be Hilbert spaces, and let $\Theta : \Gamma \otimes \mathcal{E} \to \Gamma \otimes \mathcal{E_*}$ be a contractive multi-analytic operator. Assume that $\Theta$ admits the following factorization:
	\[
	\Theta = \Theta_k\cdots\Theta_1,
	\]
	where,  for each $i=1,\dots,k$, $\Theta_i : \Gamma \otimes \mathcal{E}_{i} \to \Gamma \otimes \mathcal{E}_{i+1}$ is a contractive multi-analytic function with $\mathcal{E}=\mathcal{E}_1$ and $\mathcal{E}_*=\mathcal{E}_{k+1}$. Furthermore, define the defect operators as follows:
	\[
	\Delta_\Theta := (I - \Theta^* \Theta)^{1/2}, \quad
	\Delta_i := (I - \Theta_i^* \Theta_i)^{1/2}, \quad
	\Delta_{*i} := (I - \Theta_i \Theta_i^*)^{1/2}, \quad \text{for } i = 1, \dots,k.
	\]
	
	Following the definition established in \cite{HM26a}, a given factorization $\Theta = \Theta_k\cdots\Theta_1$ is said to be \emph{$k$-regular} if its associated linear isometry
	\[
	Z_k : \overline{{\Delta(\Gamma\otimes\mathcal{E})} }\to \overline{{\Delta_k (\Gamma\otimes\mathcal{E}_k)}} \oplus\cdots\oplus \overline{\Delta_1 (\Gamma\otimes\mathcal{E}_1)}
	\]
	defined by the relation
	\begin{equation}\label{k_regular_theta}
		Z(\Delta_\Theta f) := \Delta_k \Theta_{k-1}\cdots\Theta_1 f \oplus\cdots\oplus\Delta_1 f \quad \text{ for all } f \in \Gamma \otimes \mathcal{E},
	\end{equation}
	is a unitary operator. Equivalently,  the following  identity holds
	\[
	\overline{\{ \Delta_k \Theta_{k-1}\cdots\Theta_1 f \oplus\cdots\oplus\Delta_1 f : f \in \Gamma \otimes \mathcal{E} \}} =  \overline{{\Delta_k (\Gamma\otimes\mathcal{E}_k)}} \oplus\cdots\oplus \overline{\Delta_1 (\Gamma\otimes\mathcal{E}_1)}.
	\]
	In \cite{Po06}, G. Popescu established a one-to-one correspondence between the joint invariant subspaces of a c.n.c.\ row contraction and the regular factorizations of its characteristic function. As a natural extension of this foundational result, the authors in \cite{HM26b} demonstrated an analogous one-to-one correspondence between the \(k\)-regular factorizations of the characteristic function of a c.n.c.\ row contraction and its chain of joint invariant subspaces.

	The following theorem extends the factorization theorem of Sz.-Nagy and Foia\c{s} \cite{NF67} for the characteristic function of a contraction in block upper triangular form to the setting of row contractions in block upper triangular form.

	\begin{theorem}{\rm(Theorem 3.2, \cite{HMS17})}
		Let $\mathcal H_1$ and $\mathcal H_2$ be Hilbert spaces, and
		\[T =\begin{bmatrix}
			A& D_{A^*}L D_B\\
			0 & B
		\end{bmatrix}:(\bigoplus_{1}^n \mathcal H_1) \oplus(\bigoplus_{1}^n\mathcal H_2)\to \mathcal H_1 \oplus \mathcal H_2,\]
		be a row contraction on $\mathcal H_1 \oplus \mathcal H_2$, where $A = (A_1,
		\cdots, A_n)$ on $\mathcal H_1$ and $B = (B_1, \cdots, B_n)$ on $\mathcal H_2$
		are row contractions, and $L\in B(\mathcal D_B, \mathcal D_{A^*})$ is a
		contraction. Then
		\[\Theta_T =(I_\Gamma \otimes \sigma_*^{-1})
		\begin{bmatrix}
			\Theta_B & 0 \\
			0 & I_{\Gamma \otimes  \mathcal D_{ L^*}}
		\end{bmatrix}
		(I_\Gamma \otimes J_L)
		\begin{bmatrix}
			\Theta_A & 0\\ 0& I_{\Gamma \otimes \mathcal D_{ L}}\end{bmatrix}
		(I_\Gamma \otimes \sigma),\] 
		where $\sigma\in B(\mathcal D_T, \mathcal D_A \oplus \mathcal D_L)$ and $\sigma_*
		\in B(\mathcal D_{T^*}, \mathcal D_{B^*} \oplus \mathcal D_{L^*})$ are unitary
		operators, and $J_L=\begin{bmatrix} L^* & D_{L}
			\\ D_{L^*} & - L
		\end{bmatrix}$ is the Julia-Halmos matrix corresponding to $L$.
	\end{theorem}
	In this context, we introduce the following notation:
	\[
	\Theta'_3 = (I_\Gamma \otimes \sigma_*^{-1})
	\begin{bmatrix}
		\Theta_B & 0 \\
		0 & I_{\Gamma \otimes \mathcal{D}_{L^*}}
	\end{bmatrix}, \quad
	\Theta'_2= (I_\Gamma \otimes J_L),\quad
	\Theta'_1=
	\begin{bmatrix}
		\Theta_A & 0 \\
		0 & I_{\Gamma \otimes \mathcal{D}_L}
	\end{bmatrix}
	(I_\Gamma \otimes \sigma).
	\]
	The objective is to investigate whether the factorization 
	\begin{equation}
		\label{fact_II} \Theta =\Theta'_3 \Theta'_2 \Theta'_1
	\end{equation}
	is $3$-regular. 
	\begin{remark}\label{remark_2}
		According to the definition of $k$-regularity provided in \eqref{k_regular_theta}, the $3$-regularity of the factorization \eqref{fact_II} is equivalent to the $2$-regularity of the factorization $\Theta = \Theta_2 \Theta_1$, where
		\[
		\Theta_2 = 
		\begin{bmatrix}
			\Theta_B & 0 \\
			0 & I_{\Gamma \otimes \mathcal{D}_{L^*}}
		\end{bmatrix}, ~ \text{and} ~ 
		\Theta_1 = (I_\Gamma \otimes J_L)
		\begin{bmatrix}
			\Theta_A & 0 \\
			0 & I_{\Gamma \otimes \mathcal{D}_L}
		\end{bmatrix}.
		\]
	\end{remark}
	\begin{theorem}
		Let $\mathcal H_1$ and $\mathcal H_2$ be Hilbert spaces, and
		\[T =\begin{bmatrix}
			A& D_{A^*}L D_B\\
			0 & B
		\end{bmatrix}:(\bigoplus_{1}^n \mathcal H_1) \oplus(\bigoplus_{1}^n\mathcal H_2)\to \mathcal H_1 \oplus \mathcal H_2,\]
		be a row contraction on $\mathcal H_1 \oplus \mathcal H_2$, where $A = (A_1,
		\cdots, A_n)$ on $\mathcal H_1$ and $B = (B_1, \cdots, B_n)$ on $\mathcal H_2$
		are row contractions, and $L\in B(\mathcal D_B, \mathcal D_{A^*})$ is a
		contraction.
		Then,  in the following cases, the factorization in equation  \eqref{fact_II} of the characteristic function \(\Theta_T\) is a $3$-regular factorization:
		
		\begin{enumerate}[\rm (i)]
			\item  If the row contraction \(B\) is  pure.
			\item If the row  contraction \(T\) is  pure .
			\item If the coupling operator \(L\)  is a proper contraction, i.e.,  \(\|L^*x\| < \|x\|\) for all nonzero \(x \in \mathcal{D}_{A^*}\) or \(\|Ly\| < \|y\|\) for all nonzero \(y \in  \mathcal{D}_{B} \).
		\end{enumerate}
		
		\begin{proof}
			According to the proposition established in \cite{NF74}, the factorization 
			$$ \Theta = \Theta_2 \Theta_1, $$
			where
			\[
			\Theta_2 = 
			\begin{bmatrix}
				\Theta_B & 0 \\
				0 & I_{\Gamma \otimes \mathcal{D}_{L^*}}
			\end{bmatrix}, \quad \text{and} \quad 
			\Theta_1 = (I_\Gamma \otimes J_L)
			\begin{bmatrix}
				\Theta_A & 0 \\
				0 & I_{\Gamma \otimes \mathcal{D}_L}
			\end{bmatrix},
			\]
			is a $2$-regular factorization if and only if 
			\[\Delta_2\{\Gamma \otimes (\mathcal{D}_B\oplus\mathcal{D}_{L^*})\}\cap\Delta_{*1}\{\Gamma \otimes (\mathcal{D}_B\oplus\mathcal{D}_{L^*})\}=\{0\}.\]
			A direct calculation yields: 	 
			\begin{align}
				\Delta_2\{\Gamma &\otimes (\mathcal{D}_B\oplus\mathcal{D}_{L^*})\}\cap\Delta_{*1}\{\Gamma \otimes (\mathcal{D}_B\oplus\mathcal{D}_{L^*})\}\nonumber\\
				=&\bigg\{\begin{bmatrix}
					\Delta_B & 0\\
					0  & 0
				\end{bmatrix}
				\begin{bmatrix}
					\Gamma\otimes \mathcal{D}_B\\ \Gamma\otimes\mathcal{D}_{L^*}
				\end{bmatrix}\bigg\}\bigcap  (I_{\Gamma}\otimes J_L) \bigg\{\begin{bmatrix}
					\Delta_{*A} & 0\\
					0  & 0
				\end{bmatrix}
				\begin{bmatrix}
					\Gamma\otimes\mathcal{D}_{A^*}\\ \Gamma\otimes \mathcal{D}_{L}
				\end{bmatrix}\bigg\} \nonumber \\
				=&\big\{\Delta_B(\Gamma\otimes\mathcal{D}_B)\oplus\{0\}\big\}\cap (I_\Gamma\otimes J_L)\{
				\Delta_{*A}(\Gamma\otimes\mathcal{D}_{A*})\oplus \{0\}
				\}.\label{eq5.3}
			\end{align}
			From the Remark \ref{remark_2}, the factorization \eqref{fact_II} is $3$-regular if and only if 
			\[\big\{\Delta_B(\Gamma\otimes\mathcal{D}_B)\oplus\{0\}\big\}\cap (I_\Gamma\otimes J_L)\{
			\Delta_{*A}(\Gamma\otimes\mathcal{D}_{A*})\oplus \{0\}
			\}=\{0\}.\]
		(i) If B is a pure row contraction, then $ \Theta_B $ is inner, it implies that $ \Delta_B=(I-\Theta_B^*\Theta_B)^{1/2}=0 $. Thus
				\[\big\{\Delta_B(\Gamma\otimes\mathcal{D}_B)\oplus\{0\}\big\}\cap (I_\Gamma\otimes J_L)\{
				\Delta_{*A}(\Gamma\otimes\mathcal{D}_{A*})\oplus \{0\}
				\}=\{0\}.\]
				Hence, the factorization \eqref{fact_II} is a $3$-regular factorization.
				
				\noindent(ii) Assume that $T$ is a pure row contraction, i.e., for all $h \in \mathcal{H}_1 \oplus \mathcal{H}_2$, we have
				\[
				\lim_{k \to \infty} \sum_{|\alpha| = k} \|T^*_{\alpha} h\|^2 = 0.
				\]
				For any $h_2 \in \mathcal{H}_2$, set $h = 0 \oplus h_2$. Then, we obtain the following:
				\begin{align*}
					\lim_{k \to \infty} \sum_{|\alpha| = k} \|T^*_{\alpha} (0 \oplus h_2)\|^2 &= 0, \\
					\lim_{k \to \infty} \sum_{|\alpha| = k} \|B^*_{\alpha} h_2\|^2 &= 0.
				\end{align*}
				Thus, the operator $B$ is a pure row contraction, and hence, by case (i), the factorization \eqref{fact_II} is a $3$-regular factorization.
				
			\noindent	(iii) From the equation \eqref{eq5.3}, we get the following:
				\begin{align*}
					\Delta_2\{\Gamma&\otimes (\mathcal{D}_B\oplus\mathcal{D}_{L^*})\}\cap\Delta_{*1}\{\Gamma \otimes (\mathcal{D}_B\oplus\mathcal{D}_{L^*})\}\\
					=&\big\{\Delta_B(\Gamma\otimes\mathcal{D}_B)\oplus\{0\}\big\}\cap\{ (I_\Gamma\otimes J_L)
					(\Delta_{*A}(\Gamma\otimes\mathcal{D}_{A*})\oplus \{0\})\}\\
					=&\big\{\Delta_B(\Gamma\otimes\mathcal{D}_B)\oplus\{0\}\big\}\cap \begin{bmatrix} 
						I_\Gamma\otimes L^* & I_\Gamma\otimes D_{L} \\
						I_\Gamma\otimes D_{L^*} & - I_\Gamma\otimes L
					\end{bmatrix}
					\begin{bmatrix}
						\Delta_{*A}(\Gamma\otimes\mathcal{D}_{A*})\\ \{0\}
					\end{bmatrix}	\\
					=&\big\{\Delta_B(\Gamma \otimes \mathcal{D}_B) \oplus \{0\} \big\}\cap \big\{(I_\Gamma \otimes L^*)x\oplus (I_{\Gamma}  \otimes D_{L^*})x~:x\in\Delta_{*A}(\Gamma \otimes \mathcal{D}_{A^*})\big\}\\
					=&\left[\Delta_B(\Gamma \otimes \mathcal{D}_B)\cap (I_\Gamma \otimes L^*)\{ \Delta_{*A}(\Gamma \otimes \mathcal{D}_{A^*})\cap (\Gamma\otimes \ker D_{L^*})\}\right]\oplus \{0\}.
				\end{align*}
				It follows that the factorization \eqref{fact_II} is $3$-regular if and only if we have the equality:
				\[\Delta_B(\Gamma \otimes \mathcal{D}_B)\cap (I_\Gamma \otimes L^*)\{ \Delta_{*A}(\Gamma \otimes \mathcal{D}_{A^*})\cap (\Gamma\otimes \ker D_{L^*})\}=\{0\}.\]
				Since \( L \) is a proper contraction, then \( \ker D_{L^*} = \{0\} \), which implies that
				\[\Delta_B(\Gamma \otimes \mathcal{D}_B)\cap (I_\Gamma \otimes L^*)\{ \Delta_{*A}(\Gamma \otimes \mathcal{D}_{A^*})\cap (\Gamma\otimes \ker D_{L^*})\}=\{0\}.\] 
				Therefore, the factorization \eqref{fact_II} is a $3$-regular factorization. \qedhere 
		\end{proof}
	\end{theorem}

	\begin{remark}
		In both the single-variable and multivariable settings, if the factorizations of the characteristic function in \eqref{fact_I} and \eqref{fact_II} are $3$-regular, then the corresponding chain of invariant (or joint invariant) subspaces $\mathcal{M}_1 \subseteq \mathcal{M}_2$ must satisfy $\mathcal{M}_1=\mathcal{M}_2$. Indeed, this is a direct consequence of the second factor in a $3$-regular factorization being unitary, together with Corollary~3.7 of \cite{HM26a} and Corollary~2.9 of \cite{HM26b}.
	\end{remark}

	\noindent\textbf{Acknowledgment.}
	The research of the first author is supported in part by the Indian Institute of Technology Goa (SEED Grant 2022/SG/KH/047) and the Anusandhan National Research Foundation (MATRICS Grant MTR/2022/000339). The second author is supported by a CSIR-SRF fellowship (File No. 09/1290(12920)/2021-EMR-I) from the Council of Scientific and Industrial Research (CSIR), India.
	% Bibliography - MODIFY THIS SECTION
	
	\bibliographystyle{abbrv}
	\bibliography{myrefs}

\end{document}